\documentclass{article}

\usepackage{xcolor}
\usepackage{titlesec}
\titleformat{\section}[block]{\Large\bfseries\filcenter}{\thesection}{1em}{}
\titleformat{\subsection}[hang]{\bfseries}{}{1em}{}
\usepackage{amssymb}
\usepackage{amsfonts}
\usepackage{ucs}
\usepackage[utf8x]{inputenc}
\usepackage[english]{babel}
\usepackage[pdftex]{graphicx}
\usepackage{cancel}
\usepackage{amsmath}
\usepackage{amsthm}

\usepackage{indentfirst}
\usepackage{hyperref}
\hypersetup{
	bookmarks=true,         % show bookmarks bar?
	unicode=false,          % non-Latin characters in Acrobat’s bookmarks
	pdftoolbar=true,        % show Acrobat’s toolbar?
	pdfmenubar=true,        % show Acrobat’s menu?
	pdffitwindow=false,     % window fit to page when opened
	pdfstartview={FitH},    % fits the width of the page to the window
	pdftitle={Finding Nonlinear Production - Consumption Equilibrium},    % title
	pdfauthor={Roman A.Polyak},     % author
	pdfsubject={NPCE},   % subject of the document
	pdfcreator={},   % creator of the document
	pdfproducer={}, % producer of the document
	pdfkeywords={}, % list of keywords
	pdfnewwindow=true,      % links in new window
	colorlinks=true,       % false: boxed links; true: colored links
	linkcolor=blue,          % color of internal links (change box color with linkbordercolor)
	citecolor=green,        % color of links to bibliography
	filecolor=red,      % color of file links
	urlcolor=blue           % color of external links
}
\newcommand{\norm}[1]{\left\lVert#1\right\rVert}
\newcommand{\inner}[1]{\left<#1\right>}
\newtheorem{theorem}{Theorem}
\newtheorem{lemma}{Lemma}

\numberwithin{equation}{section}

\usepackage{float}

\title{Finding Nonlinear Production - Consumption Equilibrium}
\author{Roman A.Polyak\\George Mason University\\Fairfax VA 22030, USA\\\href{mailto:rpolyak@gmu.adu}{rpolyak@gmu.edu}}

\DeclareMathOperator{\Rset}{\mathbb{R}}
\DeclareMathOperator*{\Amax}{Argmax}
\DeclareMathOperator*{\Amin}{Argmin}
\DeclareMathOperator{\Lagr}{\mathcal{L}}

\begin{document}
	
	\selectlanguage{english}
	\graphicspath{{/}}
	\maketitle
	\nocite{Antipin, Auslender, Bakushinskiy, Censor, Dantzig, Goldstein, Kantorovich, Konnov, Koopmans, Korpelevich, Kuhn, Lancaster, Leontief, Levitin, Morishima, Polyak_res, Polyak_IO, Primak, Rosen}
\section*{Abstract}
	We introduce and study nonlinear production - consumption equilibrium (NPCE). The NPCE is a combination and generalization of both classical linear programming (LP) and classical input-output (IO) models. 
	\par In contrast to LP and IO the NPCE has both production and consumption components. Moreover, the production cost, the consumption and the factors (resources) availability are not fixed. Instead they are corespondent functions of the production output, prices of goods and prices of factors. 
	\par At the NPCE the total production cost reaches its minimum, while the total consumption, without factors expenses, reaches its maximum. At the same time the production cost is consistent with the production output, the consumption is consistent with the prices for goods and the factors availability is consistent with prices for factors. 
	\par Finding NPCE is equivalent to solving a variational inequality (VI) with a particular nonlinear operator and a simple feasible set $\Omega=\Rset_+^n\times\Rset_+^n\times\Rset_+^m$.
	\par Under natural assumptions on the production, consumption and factor operators the NPCE exists and it is unique.
	\par Projection on the feasible set $\Omega$ is a law cost operation, therefore for solving the VI we use two projection methods.
	\par Each of them requires at each step few matrix by vector multiplications and allows, along with convergence and convergence rate, establish complexity bound. 
	\par The methods decompose the problem, so both the primal and the dual variables are computed simultaneously. 
	\par On the other hand, both methods are pricing mechanisms for establishing NPCE, which is a generalization of Walras - Wald equilibrium (see \cite{Kuhn}) in few directions.  

\section{Introduction}
	For several decades linear programming (LP), introduced by Leonid V.Kantorovich before WW2, was the main instrument for optimal resources allocation (see \cite{Dantzig,Kantorovich,Koopmans} and references therein).
	\par In 1975 L.V.Kantorovich and T.C.Koopmans shared the Nobel Prize in Economics "for their contribution to the theory of optimal allocation limited resources".
	\par Also before the WW2 Wassily W.Leontief introduced the input-output (IO) model (see \cite{Leontief}). The model has been widely used for planning production, analysis economic activities, international trade, just to mention a few. 
	\par W.Leontief received the Nobel Prize in economics in 1973 "for development the input-output model and for its application to important economic problems". For quite some time LP and IO models were used independently. 
	\par Systematic effort to combine them has been undertaken by Michio Morishima more then 50 years ago. His Walras - Leontief's (WL) model consists of three systems of equations  (see \cite{Morishima}).
	\par The first system requires balance between supply and demand for each product. The second system requires balance between availability and the required consumption of each factor. The third system requires balance between the price for each produced item and total expenses required for its production. 

	\par The solution of the systems defines the WL equilibrium. Unfortunately, the assumptions on the supply and demand functions, under which WL equilibrium exists, is hard to justify from the economic standpoint. According  to M.Morishima "there is no basis for the supply and demand functions to satisfy the assumptions assumed". 
	\par The NPCE is a generalization of Leontief-Walras model (see \cite{Primak}). The assumptions, under which the NPCE exists, are justifiable from the economic stand point.
	\par The NPCE model has two parts: the production and the consumption. The production part transforms factors (resources) into products. The consumption part uses the production output to maximize the consumption. 
	\par We start by describing the building blocks of the NPCE.  
	
	\begin{enumerate}
		\item Let $A:\Rset^n\rightarrow\Rset^n$ be the balance matrix with elements $a_{ij}$, which show how much of product $1\leq i\leq n$ needed to produce one item of product $1\leq j\leq n $. We assume that $A$ has no zero rows and columns. For consumption vector $y$ we have 
		\begin{equation*}
			y=x-Ax=(I-A)x,
		\end{equation*}

	where $x=(x_1,\dots,x_n)^T$ is the production vector with components $x_j$ - the number of items of product $1\leq j\leq n$ produced. We assume that the economy is productive. It means that for any consumption vector $c\in\Rset_+^n$ the following system
	\begin{equation}\label{eq:0.1}
		(x-Ax)=c, \;\;\;\;\;\; x\in\Rset_+^n
	\end{equation}
	has a solutions. In such case matrix A is called productive matrix, (see \cite{Gale, Lancaster}).
	\par Let $\lambda$ be the price vector, i.e, $\lambda_j$ defines the price of one item of product $1\leq j\leq n$, then $q=(I-A)^T\lambda$ is the profit vector, i.e. $q_j,\;1\leq j\leq n$ is the profit out of one item of product $1\leq j\leq n$, that is
	\begin{equation}\label{eq:0.2}
		q_j=\left((I-A)^T\lambda\right)_j,\;1\leq j\leq n,\;\lambda\in\Rset_+^n.
	\end{equation}
	\par Solution of system (\ref{eq:0.1}) defines such production vector $x_c$, which guarantee consumption given by $c\in\Rset_+^n$.
	\par Solution of system (\ref{eq:0.2}) defines such price vector $\lambda_q\in\Rset_+^n$, which guarantee profit given by $q\in\Rset_+^n$.
	\par For a productive matrix $A$ the inverse $(I-A)^{-1}$ exists and it is a non-negative matrix, so both systems (\ref{eq:0.1}) and (\ref{eq:0.2}) have a unique non-negative solutions $x_c$ and $\lambda_q$ (see, for example, \cite{Gale} and \cite{Lancaster}). The productivity assumption we will keep throughout the paper. 
	\item 
	Let $B: \Rset^n\rightarrow R^m$ be the technological matrix. The elements $b_{ij}$ show how much of factor $1\leq i\leq m$ needed to produce one item of product $1\leq j\leq n$. We assume that matrix $B$ has no zero row and zero column. 
	\par 	The fundamental difference between LP, IO and NPCE is: instead of fixed production cost vector $p$, consumption vector $c$ and availability vector r we used correspondent operators.
	\item First, the fixed production cost vector $p\in\Rset_+^n$ is replaced by production operator $p:\Rset_+^n\rightarrow\Rset_{+}^n$, which maps the output $x=(x_1,\dots,x_n)^T$ into production per item cost vector $p(x)=\left(p_1(x),\dots,p_n(x)\right)^T$.
	\item Second, the fixed consumption vector $c\in\Rset_+^n$ is replaced by consumption operator $c: \Rset_+^n \rightarrow \Rset_{+}^n$, which maps the price for goods vector $\lambda=\left(\lambda_1,\dots.\lambda_n\right)^T \in\Rset_+^n$ into consumption vector\\ $c(\lambda)=\left(c_1(\lambda),\dots,c_n(\lambda)\right)^T$
	\item Third, the fixed resource vector $r\in\Rset_+^m$ is replaced by resource operator $r:\Rset_+^m\rightarrow\Rset_+^m$ which maps the price per unit factor vector $v=(v_1,\dots,v_m)^T$ into availability vector $r(v)=\left(r_1(v),\dots,r_m(v)\right)^T$.
	\end{enumerate}

\section{Nonlinear Production-Consumption Equilibrium}
The NPCE is defined as a triple $y^*=\left(x^*,\lambda^*,v^*\right)\in\Omega$: 
\begin{equation}	\label{eq:1.1}
	\begin{aligned}
		x^*\in\Amin\left\{\left<p(x^*),X\right>\Big\vert (I-A)X\geq c(\lambda^*),BX\leq r(v^*), X\in\Rset_+^n\right\},
	\end{aligned}
\end{equation}
where $X$ - production vector;
\begin{equation}	\label{eq:1.2}
	\begin{aligned}
		(\lambda^*,v^*)\in\Amax\Big\{\left<c(\lambda^*),\Lambda\right>-\left<r(v^*),V\right>\Big\vert
		(I-A)^T\Lambda-B^T V\leq p(x^*),\Lambda \in \Rset_+^n, V\in \Rset_+^m \Big\}
	\end{aligned},
\end{equation}
where $\Lambda$ - consumption price vector and $V$ - factors price vector.
\par In other words the NPCE defines such production vector $x^*\in \Rset_+^n$, consumption price vector $\lambda^*$ and factor price vector $v^*\in\Rset_+^m$, that the total production cost $\left(p\left(x^*\right), x^*\right)$ reaches its minimum, while the consumption without factors cost reaches its maximum. 
\par Moreover, the production cost per unit vector $p\left(x^*\right)$ is consistent with production output $x^*$. The consumption $c\left(\lambda^*\right)$ is consistent with the consumption prices $\lambda^*$. The factor availability $r\left(v^*\right)$ is consistent with factors price vector $v^*$.
\par Keeping in mind that under fixed $y^*=\left(x^*,\lambda^*, v^*\right)$ problems (\ref{eq:1.1}) and (\ref{eq:1.2}) are just primal and dual LP, we obtain
\begin{equation}\label{eq:1.3}
	\left<c(\lambda^*),\lambda^*\right> = \left<p(x^*), x^*\right>+\left<r(v^*),v^*\right>.
\end{equation}
\par Thus, the total consumption is equal to the sum of the production and the factors costs.
\par The NPCE is a generalization of Leontief-Walras equilibrium (see \cite{Primak}). It combines nonlinear equilibrium for resource allocation \cite{Polyak_res} with nonlinear input-output equilibrium \cite{Polyak_IO}.
\par From the complementarity conditions for all $1\leq i\leq m$ and all $1\leq j\leq n$ follows 
	\begin{equation}\label{eq:1.4}
		\begin{aligned}
				x_j^*>0&\Rightarrow p_j(x^*)+\left(B^Tv^*\right)_j-\left((I-A)^T\lambda^*\right)_j=0\\
			x_j^*=0&\Leftarrow p_j(x^*)+\left(B^Tv^*\right)_j-\left((I-A)^T\lambda^*\right)_j>0,
		\end{aligned}
	\end{equation}
	\begin{equation}\label{eq:1.5}
	\begin{aligned}
			\lambda_j^*>0&\Rightarrow\left((I-A)x^*\right)_j-c_j(\lambda^*)=0\\
		\lambda_j^*=0&\Leftarrow\left((I-A)x^*\right)_j-c_j(\lambda^*)>0,\\
	\end{aligned}
\end{equation}
	\begin{equation}\label{eq:1.6}
	\begin{aligned}
			v_i^*>0&\Rightarrow r_i(v^*)-\left(Bx^*\right)_i=0\\
		v_i^*=0&\Leftarrow r_i(v^*)-\left(Bx^*\right)_i>0.
	\end{aligned}
\end{equation}

\par In other words, the NPCE leads to the market clearing. Thus, the main issue at this point is existence of NPCE and its uniqueness.

\section{Existence of NPCE}
\par We start with conditions on operators $p, c$ and $r$, which guarantee existence and uniqueness of NPCE.
\par Let $p:\Rset_+^n\rightarrow\Rset_+^n$ be strongly monotone increasing operator, that is there exist $\alpha>0$, that for any $x_1$ and $x_2$ from $\Rset_+^n$ follows 
\begin{equation}\label{eq:2.1}
	\left<p(x_1)-p(x_2), x_1-x_2\right>\geq \alpha \norm{x_1-x_2}^2,
\end{equation}
where $\norm{x}=\left<x,x\right>^{1/2}$ is the Euclidean norm. 
\par Assumption (\ref{eq:2.1}), in particular, implies: the increase of production of any good, when the production of the rest is fixed leads to the cost per item increases. Moreover, the lower bound of the margin is $\alpha>0$.
\par Let $c:\Rset_+^n\rightarrow\Rset_+^n$ be strongly monotone decreasing operator that is there exists $\beta>0$, that for any $\lambda_1$ and $\lambda_2$ from $\Rset_+^n$ follows
\begin{equation}\label{eq:2.2}
	\left<c(\lambda_1)-c(\lambda_2),\lambda_1-\lambda_2\right>\leq -\beta\norm{\lambda_1-\lambda_2}^2.
\end{equation}
\par Assumption (\ref{eq:2.2}), in particular, implies: price increase for an item of any given product, when prices for the rest are fixed, leads to the consumption decrease of the product. Moreover, the upper bound of the margin is $-\beta<0$. 
\par Finally, let $r:\Rset_+^m\rightarrow \Rset_+^m$ be strongly monotone increasing operator, then there exists $\gamma>0$ that for any $v_1$ and $v_2$ from $\Rset_+^m$ follows
\begin{equation}\label{eq:2.3}
	\left<r(v_1)-r(v_2), v_1-v_2\right>\geq \gamma \norm{v_1-v_2}^2.
\end{equation}
\par Assumption (\ref{eq:2.3}), in particular, implies: the increase of the price for any given factor, when the prices for the rest are fixed, leads to availability increase of the factor. Moreover, the lower bound of the margin is $\gamma>0$.
\par We start by showing that finding $y^*$ from (\ref{eq:1.1})-(\ref{eq:1.2}) is equivalent to solving the following variation inequality (VI)
\begin{equation}\label{eq:2.4}
	\inner{g(y^*), y-y^*}\leq 0,\; \forall y\in \Omega,
\end{equation}
where operator $g:\Rset^{2n+m}\rightarrow\Rset^{2n+m}$ is defined as follows 
\begin{equation}\label{eq:2.5}
	\begin{aligned}
		g(y)=\Big((I-A)^T\lambda-p(x)-B^Tv;c(\lambda)-(I-A)x;Bx-r(v)\Big).
	\end{aligned}
\end{equation}
\begin{theorem}\label{thr:1}
	For $y^*=\left(x^*, \lambda^*, v^*\right)$ to be a solution of the dual LP pair (\ref{eq:1.1}) and (\ref{eq:1.2}) it is necessary and sufficient for $y^*$ to be a solution of the VI (\ref{eq:2.4}) with operator $g$ given by (\ref{eq:2.5}).
\end{theorem}
\begin{proof}

	Let consider the Lagrangian for LP (\ref{eq:1.1})
	\begin{equation*}
		\Lagr(y^*,X,\Lambda,V)=\left<p(x^*),X\right>-\left<\Lambda,(I-A)X-c(\lambda^*)\right>-\left<V,-BX+r(v^*)\right>.
	\end{equation*}
Then, 
	\begin{equation}\label{eq:2.6}
	y^*\in\Amin_{\substack{\;\;\;\;\;\;\;\;X\in\Rset_+^n}}\max_{\substack{\Lambda\in\Rset_+^n\\V\in\Rset_+^m}} \Lagr(y^*;X,\Lambda,V).
   \end{equation}
In other words, 
$$x^*\in\Amin\left\{\Lagr\left(y^*; X, \lambda^*, v^*\right)\Big\vert X\in\Rset_+^n\right\}$$
$$=\Amin\left\{\inner{p(x^*)-(I-A)^T\lambda^*+B^T\lambda^*,X}\Big\vert X\in\Rset_+^n\right\}$$
$$=\Amax\left\{\inner{(I-A)^T\lambda^*-p(x^*)-B^T\lambda^*,X}\Big\vert X\in\Rset_+^n\right\}.$$
Therefore, 
\begin{equation}\label{eq:2.7}
	(I-A)^T\lambda^*-p(x^*)-B^T\lambda^*\leq 0.
\end{equation} 
On the other hand, from (\ref{eq:1.2}) follows 
$$(\lambda^*,v^*)\in\Amax\left\{\Lagr\left(y^*;x^*,\Lambda,V\right)\Big\vert \Lambda\in\Rset_+^n, V\in\Rset_+^m\right\}$$
$$=\Amax\left\{\inner{c(\lambda^*)-(I-A)x^*,\Lambda}+\inner{Bx^*-r(v^*),V}\Big\vert\Lambda\in\Rset_+^n, V\in\Rset_+^m\right\}.$$
Therefore, 
\begin{equation}\label{eq:2.8}
	c(\lambda^*)-(I-A)x^*\leq 0,\;\;\;\;\; Bx^*-r(v^*)\leq 0.
\end{equation}
The complementarity condition for the dual LP pair (\ref{eq:1.1}) and (\ref{eq:1.2}) is given by 
\begin{equation}\label{eq:2.9}
	\inner{y^*,g(y^*)}=0.
\end{equation}
From (\ref{eq:2.7})-(\ref{eq:2.9}) follows (\ref{eq:2.4}).\\
Conversely, let us assume existence $\bar{y}\in\Omega$, that 
$$\inner{g(\bar{y}),y-\bar{y}}\leq 0,\;\;\;\;\forall y\in\Omega,$$
then
\begin{equation}\label{eq:2.10}
	\inner{g(\bar{y}),y}\leq \left<g\left(\bar{y}\right),\bar{y}\right>,\;\;\;\;\forall y\in\Omega.
\end{equation}

From ($\ref{eq:2.10}$) follows $g(\bar{y})\leq 0$, otherwise the left hand side can be made as large as one wants by taking correspondent component of $y$ large enough, while the right hand side of (\ref{eq:2.10}) is fixed. Therefore, 
\begin{align}
	\label{eq:2.11}
	(I-A)^T\bar\lambda-p(\bar x)-B^T\bar{\lambda}\leq 0,\;\;\;\;&\bar{x}\geq0,\\
	\label{eq:2.12}
	c\left(\bar\lambda\right)-(I-A)\bar{x}\leq 0,\;\;\;\;&\bar{\lambda}\geq 0,\\
	\label{eq:2.13}
	B\bar{x}-r\left(\bar v\right)\leq 0,\;\;\;\;&\bar v\geq 0.
\end{align}
From (\ref{eq:2.11}) follows that $\bar x$ is a feasible solution of the primal LP (\ref{eq:1.1}) with $y^*$ replaced by $\bar y$.\\
From (\ref{eq:2.12})-(\ref{eq:2.13}) follows that $\left(\bar \lambda, \bar v\right)$ is a feasible solution of the dual (\ref{eq:1.2}) with $y^*$ replaced by $\bar y$.\\
Also for $y=0$ from (\ref{eq:2.10}) follows
$$\left<g(\bar y),\bar y\right>\geq 0,$$
which together with (\ref{eq:2.11})-(\ref{eq:2.13}) leads to 
\begin{equation}\label {eq:2.14}
	\left<g(\bar y),\bar{y}\right>= 0
\end{equation}
Hence, $\bar{x}$ is a feasible solution of the primal LP
\begin{equation}\label{eq:2.15}
	\left<p(\bar x),\bar x\right>= \min\left\{\left<p(\bar x),X\right>\Big\vert (I-A)X\geq c\left(\bar{\lambda}\right);-BX\geq -r(\bar v); X\in \Rset_+^n\right\}
\end{equation}
and $\left<\bar \lambda,\bar v\right>$ is feasible solution of the dual LP
\begin{equation}\label{eq:2.16}
	\left<c(\bar\lambda), \bar\lambda\right>-\left<r(\bar v), \bar v\right>=\max \left\{\left<c(\bar\lambda),\Lambda\right>-\left<r(\bar v), V\right>\Big\vert (I-A)^T\Lambda\leq B^TV+p(\bar x), \Lambda\in\Rset_+^n, V\in\Rset_+^m\right\}.
\end{equation}
Thus, for the primal feasible vector $\bar x$ and the dual feasible vector $\left(\bar\lambda;\bar v\right)$ the complementarity conditions satisfied, therefore $\bar x$ solves the primal LP (\ref{eq:2.15}),  $\left(\bar\lambda;\bar v\right)$  solves the dual LP (\ref{eq:2.16}) and $\bar{y}=y^*$
	\end{proof}
Finding a saddle point (\ref{eq:2.6}) of the Lagrangian $\Lagr\left(y^*,X,\Lambda,V\right)$ is equivalent to finding an equilibrium of two person concave game with following payoff functions.\\
 For the first player the payoff function is 
\begin{equation*}
	\varphi_1 (y;X, \lambda,v) =-\Lagr (y;X,\lambda, v) = \left<-p(x)+(I-A)\lambda -B^Tv, X\right>+\left<v,r(v)\right>
\end{equation*}
with strategy $X\in\Rset_+^n$\\
For the second player the payoff function is 
\begin{equation*}
	\varphi_2(y;x,\Lambda,V)=\Lagr (y;x,\Lambda, V)=\left<c(\lambda)-(I-A)x,\Lambda\right>+\left<Bx-r(v),V\right>+\left<p(x),x\right>
\end{equation*}
with the strategy $(\Lambda,V)$.\\
Now let us consider the normalized payoff function $\Phi: \Omega\time\Omega\rightarrow\Rset$, given by the following formula
$$\Phi\left(y,Y\right)=\varphi_1(y;X,\lambda,v)+\varphi_2(y;x,\Lambda,V).$$
We also consider the following map
\begin{equation}\label{eq:2.17}
	y\rightarrow \omega(y)=\Amax\left\{\Phi\left(y,Y\right)\Big\vert Y\in\Omega\right\}.
\end{equation}
Then, finding $y^*=\left(x^*,\lambda^*,v^*\right)$ is equivalent to finding a fixed point of the map $y\rightarrow \omega(y)$.\\
The normalized payoff function $\Phi\left(y,Y\right)$, however,  is linear in $Y$ and $\Omega=\Rset_+^n\times\Rset_+^n\times\Rset_+^m$ is an unbounded set, therefore, one can't use the map (\ref{eq:2.17}) for proving existence and uniqueness of $y^*$.\\
We will do it later by using assumption (\ref{eq:2.1}) - (\ref{eq:2.3}).
\par Meanwhile, we would line to point out that the nonlinear operator $g$ defined by (\ref{eq:2.5}) is, in fact, the pseudo-gradient of $\Phi\left(y,Y\right)$, i.e. 
\begin{equation*}
	g(y)=\nabla_Y \Phi\left(y,Y\right)\Big\vert_{Y=y}=\Big((I-A)^T\lambda-p(x)-B^Tv;c(\lambda)-(I-A)x;Bx-r(v)\Big).
\end{equation*}
\par The following Lemma establishes sufficient condition for pseudo-gradient $g:\Omega\rightarrow \Rset_+^{2n+m}$ to be strongly monotone.
\begin{lemma}\label{lm:1}
	If operators $p,c$ and $r$ strongly monotone, i.e. (\ref{eq:2.1})-(\ref{eq:2.3}) hold, then the pseudo - gradient $g:\Omega\rightarrow\Rset^{2n+m}$ given by (\ref{eq:2.5}) is also strongly monotone, i.e. for $\delta=\min\{\alpha,\beta,\gamma\}>0$, and any $y_1$ and $y_2$ from $\Omega$, we have 
	\begin{equation}\label{eq:2.18}
		\left<g(y_1)-g(y_2),y_1-y_2\right>\leq-\delta\norm{y_1-y_2}^2.
	\end{equation}

\end{lemma}
\begin{proof}
	Let $y_1=(x_1,\lambda_1,v_1)$ and $y_2=(x_2,\lambda_2,v_2)$ be vectors from $\Omega$, then 

\begin{multline*}
	\left<g(y_1)-g(y_2),y_1-y_2\right>=\left<-p(x_1)+(I-A)^T\lambda_1,-B^Tv_1 + p(x_2)-(I-A)^T\lambda_2+B^Tv_2, x_1-x_2\right>+\\
	\left<c(\lambda_1)-(I-A)x_1-c(\lambda_2)+(I-A)x_2,\lambda_1-\lambda_2\right>+\\
	\left<Bx_1-r(v_1)-Bx_2+r(v_2),v_1-v_2\right>=
\end{multline*}
\begin{multline*}
	=-\left<p(x_1)-p(x_2),x_1-x_2\right>+\\\left<(I-A)^T(\lambda_1-\lambda_2), x_1-x_2\right>-
	\left<B^T(v_1-v_2),x_1-x_2\right>\\
	+\left<c(\lambda_1)-c(\lambda_2),\lambda_1-\lambda_2\right>-\left<(I-A)(x_1-x_2),\lambda_1-\lambda_2\right>\\
	+\left<B(x_1-x_2),v_1-v_2\right>-\left<r(v_1)-r(v_2), v_1-v_2\right>\\
	=-\left<p(x_1)-p(x_2),x_1-x_2\right>+\left<c(\lambda_1)-c(\lambda_2),\lambda_1-\lambda_2\right>-\left<r(v_1)-r(v_2),v_1-v_2\right>.
\end{multline*}
From (\ref{eq:2.1})-(\ref{eq:2.3}) follows 
$$\left<g(y_1)-g(y_2),y_1-y_2\right>\leq -\alpha\norm{x_1-x_2}^2-\beta\norm{\lambda_1-\lambda_2}^2-\gamma\norm{v_1-v_2}^2.$$
Therefore, for $\delta=\min\{\alpha,\beta,\gamma\}>0$ we obtain (\ref{eq:2.18}).
\end{proof}
\begin{theorem}\label{thr:2}
	If $p,c$ and $r$ are continuous and strongly monotone operators, that is (\ref{eq:2.1})-(\ref{eq:2.3}) hold, than the NPCE $y^*=(x^*,\lambda^*,v^*)$ exists and it is unique.
\end{theorem}
\begin{proof}
	Let consider $y_0\in\Rset_{++}^{2n+m}$  such that $\norm{y_0}\leq 1$ and large enough number $M>0$. We  replace $\Omega$ by $\Omega_M=\left\{y\in\Omega:\norm{y-y_0}\leq M\right\}$.
	\par Then, $\Omega_M$ is a bounded, convex and closed set and so is the set
	\begin{equation}\label{eq:2.19}
		\omega_M(y)=\Amax\left\{\Phi\left(y,Y\right) \vert Y\in\Omega_M\right\},
	\end{equation}
	because for any given $y\in\Omega$ function $\Phi\left(y,Y\right)$ is linear in $Y$.
	\par Moreover, due to the continuity of $p,c$ and $r$ the map $y\rightarrow\omega_M(y)$ is upper semi-continuous. 
	\par Therefore, from the Kakutani's theorem follows existence of $y_M^*\in\Omega_M: y_M^*\in\omega\left(y_M^*\right)$ (see, for example, \cite{Lancaster}).

\par 
It turns out, the constraint $\norm{y-y_0}\leq M$ is irrelevant in (\ref{eq:2.19}). In fact, from (\ref{eq:2.18}) with $y_1=y_0$ and $y_2=y_M^*$ follows 
\begin{equation}\label{eq:2:20}
	\delta\norm{y_0-y_M^*}^2\leq \left<g\left(y_m^*\right)-g\left(y_0\right), y_0-y_M^*\right>=
	\left<g\left(y_M^*\right),y_0-y_M^*\right>+\left<g\left(y_0\right), y_M^*-y_0\right>
\end{equation}
From (\ref{eq:2.19}) and $y_M^*\in\omega\left(y_M^*\right)$ follows $\left(g\left(y_M^*\right),y_0-y_M^*\right)\leq 0$, therefore from (\ref{eq:2:20}) we obtain 
$$\delta\norm{y_0-y_M^*}^2\leq \norm{g(y_0)}\norm{y_M^*-y_0},$$
or
	$$\norm{y_0-y_M^*}\leq \delta^{-1} \norm{g(y_0)}.$$
\par Therefore, $\norm{y_M^*}\leq \norm{y_0}+\norm{y_0-y_M^*}\leq 1+\delta^{-1}\norm{g(y_0)}$. Thus, for large enough $M>0$ the constraint $\norm{y_0-y}\leq M$ can't be active in (\ref{eq:2.19}), therefore it is irrelevant in (\ref{eq:2.19}) and there is $y^*\in\Omega:y^*\in\omega(y^*)$. Uniqueness follows from (\ref{eq:2.18}). In fact, assuming that there are two vectors $y^*$ and $y^{**}$, which solve VI (\ref{eq:2.4}), we obtain 
$$\left<g(y^*), y^{**}-y^*\right>\leq 0 \text{  and  } \left<g(y^{**}), y^*-y^{**}\right>\leq 0,$$
therefore 
$$\left<g\left(y^*\right)-g\left(y^{**}\right), y^{**}-y^*\right>\leq 0$$
or
$$\left<g\left(y^{*}\right)-g\left(y^{**}\right) y^{*}-y^{**}\right>\geq 0,$$
which contradicts (\ref{eq:2.18}) with $y_1=y^*, y_2=y^{**}$. Thus, assumptions (\ref{eq:2.1})-(\ref{eq:2.3}) are sufficient for existence and uniqueness NPCE $y^*=(x^*,\lambda^*,v^*)$.
\end{proof}
\par In what is following we replace the global strong monotonicity assumptions (\ref{eq:2.1}) - (\ref{eq:2.3}) for strong monotonicity only at the NPCE $y^*=(x^*,\lambda^*,v^*)$, i.e we assume existence of $\alpha>0, \beta>0$ and $\gamma>0$, that 
\begin{equation}\label{eq:2.21}
	\left<p(x)-p(x^*),x-x^*\right>\geq \alpha\norm{x-x^*}^2,\;\;\;\forall x\in\Rset_+^n
\end{equation}
\begin{equation}\label{eq:2.22}
	\left<c(\lambda)-c(\lambda^*),\lambda-\lambda^*\right>\leq -\beta \norm{\lambda-\lambda^*}^2,\;\;\;\forall\lambda\in\Rset_+^n
\end{equation}
\begin{equation}\label{eq:2.23}
	\left<r(v)-r(v^*),v-v^*\right>\geq \gamma\norm{v-v^*}^2,\;\;\;\forall v\in \Rset_+^n
\end{equation}
hold. 
\par In the next section we also assume that operator $p,c$ and $r$ satisfy local Lipschitz condition, i.e there exist $L_p>0, L_c>0$ and $L_r>0$ that 
\begin{equation}\label{eq:2.24}
	\norm{p(x)-p(x^*)}\leq L_p\norm{x-x^*},\;\;\;\forall x\in\Rset_+^n
\end{equation}
\begin{equation}\label{eq:2.25}
	\norm{c(\lambda)-c(\lambda^*)}\leq L_c\norm{\lambda-\lambda^*},\;\;\;\forall \lambda\in\Rset_+^n
\end{equation}
\begin{equation}\label{eq:2.26}
	\norm{r(v)-r(v^*)}\leq L_r\norm{v-v^*},\;\;\;\forall v\in\Rset_+^n
\end{equation}
hold.
\par We will say that operator $p,c$ and $r$ are well defined if conditions (\ref{eq:2.21})-(\ref{eq:2.26}) are satisfied. 
\par Assumption (\ref{eq:2.21}) means that production operator $p$ is sensitive to the change of the production output only at $x^*$.
\par Assumption (\ref{eq:2.22}) means that consumption operator $c$ is sensitive to the consumption price change only at $\lambda^*$. 
\par Finally,  from (\ref{eq:2.23}) follows that operator $r$ is sensitive to the factor price change only at $v^*$.
\par  Conditions (\ref{eq:2.24}) - (\ref{eq:2.26}) mean that production, consumption and factors availability are under control in the neighborhood of equilibrium. 
\section{Pseudo-Gradient Projection Method for finding NPCE}
The projection method for solving optimization problems was introduced in the 60s (see \cite{Goldstein, Levitin, Rosen}).
\par Later projection methods were used for solving VI problems (see, for example, \cite{Antipin,Auslender, Bakushinskiy, Censor, Konnov, Polyak_res, Polyak_IO} and references therein). Often projection methods have only theoretical value because projection on a convex set in $\Rset^n$ is numerically costly operation. In case of simple feasible set, however, projection methods can be efficient (see, for example, \cite{Polyak_res, Polyak_IO}). 
\par From Theorem \ref{thr:1} follows: for finding NPCE $y^*$ it is sufficient to solve VI (\ref{eq:2.4}) with operator $g$ given by (\ref{eq:2.5}). 
\par The feasible set $\Omega$ a very simple and projection on it is a low cost operation. Therefore, for finding $y^*=(x^*,\lambda^*, v^*)$ we use two projection methods. 
\par The projection of $u\in\Rset^q$ on $\Rset_+^q$ is given by the following formula 
$$v=P_{\Rset_+^q}(u)=\left[u\right]_+=\left(\left[u_1\right]_+,\dots,\left[u_q\right]_+\right),$$
where for any $1\leq i\leq q$ we have 
$$\left[u_j\right]_+=\begin{cases}
	u_i\;\;,&u_i\geq 0\\
	0\;\;,&u_i<0
\end{cases}.$$
\par The projection operator $P_\Omega:\Omega\rightarrow\Omega$ is given by formula 
$$P_\Omega=\left(\left[x\right]_+,\left[\lambda\right]_+,\left[v\right]_+\right),$$
where $y=(x,\lambda,v)\in\Rset^{2n+m}$.
\par The following well known properties of operator $P_\Omega$ will be used later. 
\par First, the operator $P_\Omega$ is not expansive, i.e. for any $u_1$ and $u_2$ from $\Omega$ we have 
\begin{equation}\label{eq:3.1}
	\norm{P_\Omega(u_1)-P_\Omega(u_2)}\leq\norm{u_1-u_2}.
\end{equation}
\par Second, vector $y^*$ solves VI (\ref{eq:2.4}) iff for any $t>0$ vector $y^*$ is a fixed point of a map\\ $P_\Omega\left(I+tg\right):\Omega\rightarrow\Omega$, i.e 
\begin{equation}\label{eq:3.2}
	y^*=P_\Omega\left(y^*+tg\left(y^*\right)\right)
\end{equation}
for any $t>0$.
\par We are ready to describe the Pseudo-Gradient Projection (PGP) method for finding NPCE $y^*$.
\par Let $y_0=(x_0,\lambda_0,v_0)\in\Omega$ be the starting point and let $y_s=(x_s,\lambda_s,v_s)\in\Omega$ has been found already.
\par The PGP method finds the next approximation by the following formulas
\begin{equation}\label{eq:3.3}
	y_{s+1}=P_\Omega\left(y_s+tg(y_s)\right).
\end{equation}
The step length $t>0$ we will specify later.
\par The PGP method (\ref{eq:3.3}) updates the primal and dual variables by the following formula
\begin{equation}\label{eq:3.4}
	x_{s+1,j}=\left[x_{s,j}+t\left((I-A)^T\lambda_s-p(x_s)-B^Tv_s\right)_j\right]_+\;\;\;\;1\leq j\leq n
\end{equation}
\begin{equation}\label{eq:3.5}
	\lambda_{s+1,j}=\left[\lambda_{s,j}+t\left(c(\lambda_s)-(I-A)x_s\right)_j\right]_+\;\;\;\;1\leq j\leq n
\end{equation}
\begin{equation}\label{eq:3.6}
	v_{s+1,i}=\Big[v_{s,i}+t\left(Bx_s-r(v_s)\right)_i\Big]_+\;\;\;\;1\leq i\leq m.
\end{equation}
In fact, the PGP is a decomposition method, which computes the primal $x_s$ and the dual $\left(\lambda_s;v_s\right)$ vectors independently.
\par On the other hand, formulas (\ref{eq:3.4})-(\ref{eq:3.6}) is a pricing mechanism for establishing NPCE $y^*$.
\par From (\ref{eq:3.4}) follows: if for any product $1\leq j\leq n$ the market value of one unit $\left[\left(I-A\right)^T\lambda_s\right]_j$ is greater then the sum of the production and factor cost
$\left(p(x_s)+\left(B^Tv_s\right)\right)_j$ then production $(x_s)_j$ has to be increased, while if
$\left[\left(I-A\right)^T\lambda_s\right]_j<\left[p(x_s)+\left(B^Tv_s\right)\right]_j$,  then production $(x_s)_j$ has to be reduced. 
\par From (\ref{eq:3.5}) follows: if for any product $1\leq j\leq n$ the current consumption $c_j\left(\lambda_s\right)$ is greater then the production output $\left[\left(I-A\right)x_s\right]_j$, then the price $\left(\lambda_s\right)_j$ has to be increased, while if $c_j(\lambda_s)<\left[\left(I-A\right)x_s\right]_j$, then the price $\left(\lambda_s\right)_j$ has to be reduced. 
\par Finally, from (\ref{eq:3.6}) follows  : if  for any factor $1\leq i\leq m$ the consumption $\left(Bx_s\right)_i$ is greater then the availability $r_i(v_s)$ of this factor under price $v_s$, then the price $(v_s)_i$ has to be increased, while if 
$\left(Bx_s\right)_i<r_i(v_s)$, then the price $(v_s)_i$ has to be reduced.
\par The PGP method (\ref{eq:3.3}) can be viewed as explicit Euler's projection method for solving the following system of ordinary differential equations 
\begin{equation}\label{eq:3.7}
	\begin{aligned}
		\frac{dx}{dt}&=(I-A)^T\lambda-p(x)-B^Tv\\\\
		\frac{d\lambda}{dt}&=c(\lambda)-(I-A)x\\\\
		\frac{dv}{dt}&=Bx-r(v)
	\end{aligned}
\end{equation}
with $x(0)=x_0,\lambda(0)=\lambda_0$ and $v(0)=v_0$.
For convergence analysis of the PGP method (\ref{eq:3.3}) we need the following Lemma similar to Lemma 
\ref{lm:1}. 
\begin{lemma}\label{lm:2}
	If operators $p,c$ and $r$ are strongly monotone at the equilibrium $y^*=(x^*,\lambda^*,v^*)$,\\ i.e (\ref{eq:2.21}) - (\ref{eq:2.23}) hold, then operator $g:\Omega\rightarrow \Rset_+^{2n+m}$ given by (\ref{eq:2.5}) is strongly monotone at $y^*$, i.e. for $\delta=\min\left\{\alpha,\beta,\gamma\right\}>0$ the following bound
	\begin{equation}\label{eq:3.8}
		\left<g(y)-g(y^*),y-y^*\right>\leq-\delta\norm{y-y*},\;\;\;\forall y\in\Omega
	\end{equation}
	holds.
\end{lemma}
The proof is similar to proof of Lemma \ref{lm:1}.
\begin{lemma}\label{lm:3}
	If operators $p,c$ and $r$ satisfy local Lipschits condition (\ref{eq:2.24})-(\ref{eq:2.26}) then there is $L>0$ such that the bound 
	\begin{equation}\label{eq:3.9}
		\norm{g(y)-g(y^*)}\leq L\norm{y-y^*}
	\end{equation}
holds.
\end{lemma}
\par The proof of the Lemma \ref{lm:3} and the upper bound for $L$ can be established  using technique similar to the one used in \cite{Polyak_res}. 
\par \textbf{Remark 1}. We assume that for given $x\in \Rset_+^n$ and $\lambda\in\Rset_+^n$ computing $p(x)$ and $c(\lambda)$ does not requires more then $\mathcal{O}(n^2)$ operations and for a given $v\in \Rset_+^m$
computing $r(v)$ does not requires more then $\mathcal{O}(m^2)$ operations.
\par It is true, for example, if $p(x)=\nabla\left(\frac{1}{2}x^TPx+q^Tx\right)$, where $P$ is symmetric positive semi-defined matrix, $c(\lambda)=\nabla\left(\frac{1}{2}\lambda^TC\lambda+d^T\lambda\right)$, where $C$ is symmetric negative semi-defined matrix in $\Rset_+^n$ and \\$r(v)=\nabla\left(\frac{1}{2}v^TRv+s^Tv\right)$, where $R$ is symmetric positive semi-defined matrix in $\Rset_+^m$.
\par In what is following we assume that for any given $x\in \Rset^n, \lambda\in \Rset^n$ and $v\in \Rset^m$ computing $p(x), c(\lambda), r(v)$ does not require more then $\mathcal{O}(n^2)$ operations, $n>m$.
\par Then, each PGP step (\ref{eq:3.3}) does not need more then $\mathcal{O}(n^2)$ operations as well, because formulas (\ref{eq:3.4}) - (\ref{eq:3.6}) require few matrix by vector multiplication.
\par Let $\varkappa=\gamma L^{-1}$ be the condition number of the VI operator $g$.
\par The following Theorem establishes global Q-linear convergence rate of the PGP method (\ref{eq:3.3}). The proof is similar to the proof of Theorem 3 in \cite{Polyak_res}. We sketch the proof for completeness. 
\begin{theorem}\label{th:3}
	If operators $p,c$ and $r$ are well defined, i.e. conditions (\ref{eq:2.21}) - (\ref{eq:2.26}) holds then:
	\begin{enumerate}
		\item for any $0<t<2\delta L^{-2}$ the PGP method (\ref{eq:3.3}) globally converges to NPCE $y^*=(x^*,\lambda^*,v^*)$ with Q-linear rate with a ratio $0<q(t)=\left(1-2t\delta+t^2L^2\right)^{1/2}<1$, i.e
		\begin{equation}\label{eq:3.10}
			\norm{y_{s+1}-y^*}\leq q(t)\norm{y_s-y^*};
		\end{equation}
		\item for $t=\delta L^{-2}=\Amin\left\{q(t)\vert t>0\right\}$ the following bound holds
		\begin{equation}\label{eq:3.11}
			\norm{y_{s+1}-y^*}\leq (1-\varkappa^2)^{1/2}\norm{y_s-y^*}=q\norm{y-y^*};
		\end{equation}
		\item the PGP complexity is 
		\begin{equation}\label{eq:3.12}
			\text{Comp}\left(PGP\right)=\mathcal{O}\left(n^2\varkappa^{-2}\ln \epsilon^{-1}\right),
		\end{equation}
		where $\epsilon>0$ is the required accuracy. 
	\end{enumerate}
\end{theorem}
\begin{proof}
	From (\ref{eq:3.1}) - (\ref{eq:3.3}) follows 
	\begin{multline}\label{eq:3.13}
		\norm{y_{s+1}-y^*}^2=\norm{P_\Omega\left(y_s+tg(y_s)\right)-P_\Omega\left(y^*+tg(y^*)\right)}^2\leq \norm{y_s+tg(y_s)-y^*-tg(y^*)}^2\leq\\
		=\norm{y_s-y^*}^2+2t\left(y_s-y^*, g(y_s)-g(y^*)\right)+t^2\norm{g(y_s)-g(y^*)}^2.
	\end{multline}
\par From Lemma \ref{lm:2}, Lemma \ref{lm:3} and ({\ref{eq:3.13}}) follows 
$$\norm{y_{s+1}-y^*}\leq \norm{y_s-y^*}\left(1-2t\delta+t^2L^2\right)^{1/2},$$
therefore for $0<t<2\delta L^{-2}$ we have $0<q(t)=\left(1-2t\delta+t^2L^2\right)^{1/2}<1$.
\par Let $0<\epsilon \ll 1$ is the required accuracy. In view of (\ref{eq:3.11}) it takes $\mathcal{O}\left(\left(\ln q\right)^{-1}\ln \epsilon\right)$ steps to find \\$\epsilon$-approximation for NPCE $y^*$. Each step requires $\mathcal{O}(n^2)$ operations. Therefore, it takes
\begin{equation}\label{eq:3.14}
	N=\mathcal{O}\left(n^2\frac{\ln\epsilon^{-1}}{\ln q^{-1}}\right)
\end{equation} 
operations for finding an $\epsilon$-approximation for NPCE $y^*$. 
\par From (\ref{eq:3.11}) follows $\left(\ln q^{-1}\right)^{-1}=\left(-\frac{1}{2}\ln(1-\varkappa^2)\right)^{-1}$, then from $\ln (1+x)\geq x,\;\forall x>-1$ follows $\left(\ln q^{-1}\right)\leq 2\varkappa^{-2}$, which together with (\ref{eq:3.14}) leads to the bound (\ref{eq:3.12}).
\end{proof}
\par If $\delta=\min\left\{\alpha,\beta,\gamma\right\}=0$, then Theorem \ref{th:3} does not guarantee even convergence of the PGP method. Therefore, we need another tool for finding \textbf{NPCE} $y^*$.
\par In the following section we apply Extra Pseudo-Gradient (EPG) method for solving VI (\ref{eq:2.4}).
\par The Extra Gradient (EG) method for finding saddle points was introduced by G.Korpelevich in the 70s (see \cite{Korpelevich}). Since then the EG method became a popular tool for solving VI (see \cite{Antipin, Censor, Konnov, Polyak_res, Polyak_IO} and references therein).
\par We assume that operators $p,c$ and $r$ are monotone and satisfy Lipschitz condition, so is the operator $g$, i.e. for any $y_1$ and $y_2$ from $\Omega$ we have 
\begin{equation}\label{eq:3.15}
	\left<g(y_1)-g(y_2),y_1-y_2\right>\leq 0
\end{equation}
and there exists $L>0$, that 
\begin{equation}\label{eq:3.16}
	\norm{g(y_1)-g(y_2)}\leq L\norm{y_1-y_2}.
\end{equation}
\section{Extra Pseudo-Gradient method for finding NPCE.}
\par Application of G.Korpelevich's extra-gradient method \cite{Korpelevich} for solving VI (\ref{eq:2.4}) leads to two phases extra pseudo-gradient (EPG) method for finding NPCE. 
\par At the prediction phase we predict the production vector $\hat{x}$, the consumption price vector $\hat{\lambda}$, and the factor price vector $\hat{v}$. Then we find the predicted production cost per unit vector $p(\hat{x})$, predicted consumption vector $c(\hat{\lambda})$ and predicted availability factor vector $r(\hat{v})$.
\par At the correction phase we use $p(\hat{x}),c(\hat{\lambda})$ and $r(\hat{v})$ to find the new
production, new consumption and new factor prices.
\par Let $y_0=(x_0,\lambda_0,v_0)\in\Omega$ be the starting point and approximation $y_s=(x_s,\lambda_s,v_s)$ has been found already.
\par At the prediction phase one finds 
\begin{equation}\label{eq:4.1}
	\hat{y}_s=P_\Omega\left(y_s+tg(y_s)\right)=\left[y_s+tg(y_s)\right]_+.
\end{equation} 
\par At the correction phase one finds the new approximation 
\begin{equation}\label{eq:4.2}
	y_{s+1}=P_\Omega\left(y_s+tg(\hat{y}_s)\right)=\left[y_s+tg(\hat{y}_s)\right]_+.
\end{equation} 
\par In other words, the EPG method first predicts the production vector 
\begin{equation}\label{eq:4.3}
	\hat{x}_s=\left[x_s+t\left((I-A)^T\lambda_s-p(x_s)-B^Tv_s\right)\right]_+,
\end{equation}
the consumption price vector 
\begin{equation}\label{eq:4.4}
	\hat{\lambda}_s=\left[\lambda_s+t\left(c(\lambda_s)-(I-A)x_s\right)\right]_+,
\end{equation}
and the factor price vector
\begin{equation}\label{eq:4.5}
	\hat{v}_s=\left[v_s+t\left(Bx_s-r(v_s)\right)\right]_+.
\end{equation}
Then, EPG method finds new approximation for the production vector
\begin{equation}\label{eq:4.6}
	x_{s+1}=\left[x_s+t\left((I-A)^T\hat{\lambda}_s-p(\hat{x}_s)-B^T\hat{\lambda}_s\right)\right],
\end{equation}
as well as, new approximation for the price vector 
\begin{equation}\label{eq:4.7}
	\lambda_{s+1}=\left[\lambda_s+t\left(c(\hat{\lambda}_s)-(I-A)\hat{x}_s\right)\right]_+,
\end{equation}
and the factor price vector
\begin{equation}\label{eq:4.8}
	v_{s+1}=\left[v_s+t\left(B\hat{x}_s-r(\hat{v}_s)\right)\right]_+.
\end{equation}
The meaning of the formula (\ref{eq:4.3})-(\ref{eq:4.5}) and (\ref{eq:4.6})-(\ref{eq:4.8}) is similar to the meaning of formulas (\ref{eq:3.4})-(\ref{eq:3.6}). The step - length $t>0$ we specify in the course of proving the following Theorem.
\par The EPG method (\ref{eq:4.3})-(\ref{eq:4.8}) is a pricing mechanism for finding NPCE, which guarantee convergence under minimum assumptions in the input data. 
\begin{theorem}\label{th:4}
	If $p,c$ and $r$ be monotone operators, which satisfy Lipschitz condition, i.e (\ref{eq:3.15}) and (\ref{eq:3.16}) holds, then for any $0<t<\left(\sqrt{2}L\right)^{-1}$ the EPG method generates a converging sequence $\left\{y_s\right\}_{s=0}^{\infty}$ and $\lim\limits_{s\rightarrow \infty}y_s=y^*$.
\end{theorem}
\begin{proof}
	Form (\ref{eq:4.1}) and (\ref{eq:4.2}), non-expansiveness of the operator $P_\Omega$ and Lipschitz condition (\ref{eq:3.16}) follows
	\begin{equation}\label{eq:4.9}
				\norm{y_{s+1}-\hat{y}_s}=\norm{P_\Omega(y_s+tg\left(\hat{y}_s\right))-P_\Omega(y_s+tg\left(y_s\right))}\leq t\norm{g(\hat{y}_s)-g(y_s)}\leq tL\norm{\hat{y}_s-y_s}.
	\end{equation} 
 From (\ref{eq:4.1}) we have 
$$\left<tg(y_s)+(y_s-\hat{y}_s) , y-\hat{y}_s)\right>\leq 0,\;\;\;\forall y\in \Omega.$$
Therefore, for $y=y_{s+1}$ we obtain 
$$\left<tg(y_s)+(y_s-\hat{y}_s) , y_{s+1}-\hat{y}_s)\right>\leq 0,$$
so
\begin{equation}\label{eq:4.10}
	\left<y_s-\hat{y}_s, y_{s+1}-\hat{y}_s\right>+t\left<g(\hat{y}_s), y_{s+1}-\hat{y}_s\right>-t\left<g(\hat{y}_s)-g(y_s), y_{s+1}-y_s\right>\leq 0.
\end{equation}
Then, from (\ref{eq:4.9}) and Lipschitz condition (\ref{eq:3.16}) we obtain 
\begin{equation}\label{eq:4.11}
	t\left<g(\hat{y}_s)-g(y_s), y_{s+1}-\hat{y}_s\right>\leq t\norm{g(\hat{y}_s)-g(y_s)}\norm{y_{s+1}-\hat{y}_s}\leq 
	(tL)^2\norm{\hat{y}_s-y_s}.
\end{equation}
From (\ref{eq:4.10}) and (\ref{eq:4.11}) follows 
\begin{equation}\label{eq:4.12}
	\left<y_s-\hat{y}_s, y_{s+1}-\hat{y}_s\right>+t\left<g(\hat{y}_s), y_{s+1}-\hat{y}_s\right>-(tL)^2\norm{\hat{y}_s-y_s}^2\leq 0.
\end{equation}
From (\ref{eq:4.2}) follows 
\begin{equation}\label{eq:4.13}
	\left<tg(\hat{y}_s)+y_s-y_{s+1}, y-y_{s+1}\right>\leq 0,\;\;\;\forall y\in\Omega.
\end{equation}
Hence, for $y=y^*$ we obtain 
\begin{equation}\label{eq:4.14}
	\left<y_s-y_{s+1}+tg(\hat{y}_s), y^*-y_{s+1}\right>\leq 0.
\end{equation}
From $\left(g(y^*), y-y^*\right)\leq 0, \;\;\;\forall y\in\Omega$ follows
$$\left<g(y^*),\hat{y}_s-y^*\right>\leq 0,$$
so for any $t>0$ we have 
\begin{equation}\label{eq:4.15}
	t\left<-g(y^*), y^*-\hat{y}_s\right>\leq 0.
\end{equation}
By adding (\ref{eq:4.12}), (\ref{eq:4.14}) and (\ref{eq:4.15}) and using monotonisity of operator $g$, i.e. 
$$\left<g(\hat{y}_s)-g(y^*), y^*-\hat{y}_s\right>\geq 0$$
we obtain 
\begin{equation}\label{eq:4.16}
	2\left<y_s-y_{s+1}, y^*-y_{s+1}\right>+2\left<y_s-\hat{y}_s, y_{s+1}-\hat{y}_s\right>-2\left(tL\right)^2\norm{\hat{y}_s-y_s}\leq 0.
\end{equation}
Form the following identity 
\begin{equation}\label{eq:4.17}
	2\left<u-v,w-v\right>=\norm{u-v}^2+\norm{v-w}^2-\norm{u-w}^2
\end{equation}
with $u=y_s, v=y_{s+1}$ and $w=y^*$ follows 
\begin{equation}\label{eq:4.18}
	2\left<y_s-y_{s+1}, y^*-y_{s+1}\right>=\norm{y_s-y_{s+1}}^2+\norm{y_{s+1}-y^*}^2-\norm{y_s-y^*}^2.
\end{equation}
Using identity (\ref{eq:4.12}) with $u=y_s, v=\hat{y}_s$ and $w=y_{s+1}$ we obtain 
\begin{equation}\label{eq:4.19}
	2\left<y_s-\hat{y}_s, y_{s+1}-\hat{y}_s\right>=\norm{y_s-\hat{y}_s}^2+\norm{\hat{y}_s-y_{s+1}}^2-\norm{y_s-y_{s+1}}^2.
\end{equation}
Then, from (\ref{eq:4.16}), (\ref{eq:4.18}) and (\ref{eq:4.19}) follows 
\begin{equation}\label{eq:4.20}
	\norm{y_{s+1}-y^*}^2+\left(1-2(tL)^2\right)\norm{y_s-\hat{y}}^2+\norm{\hat{y}_s-y_{s+1}}^2\leq \norm{y_s-y^*}^2.
\end{equation}
Summing up (\ref{eq:4.20}) from $s=0$ to $s=N$ we obtain 
\begin{equation}\label{eq:4.21}
	\norm{y_{_N+1}-y^*}^2+\left(1-2(tL)^2\right)\sum_{s=0}^{N}\norm{y_s-\hat{y}_s}^2+
	\sum_{s=0}^{N}\norm{\hat{y}_s-y_{s+1}}^2\leq \norm{y_0-y^*}^2.
\end{equation}
From (\ref{eq:4.21}) and $0<t< \left(\sqrt{2}L\right)^{-1}$ we obtain
\begin{equation}\label{eq:4.22}
	\sum_{s=0}^{\infty} \norm{y_s-\hat{y}_s}^2<\infty \;\;\text{and}\;\;\sum_{s=0}^{\infty} \norm{\hat{y}_s-y_{s+1}}^2 <\infty.
\end{equation}
Therefore, 
\begin{equation}\label{eq:4.23}
	\text{(a)}\;\lim\limits_{s\rightarrow\infty}\norm{y_s-\hat{y}_s}=0,\;\;\text{(b)}\; \lim\limits_{s\rightarrow\infty}\norm{y_s-y_{s+1}}=0.
\end{equation}
Also from (\ref{eq:4.20}) and $0<t< \left(\sqrt{2}L\right)^{-1}$ follows 
\begin{equation}\label{eq:4.24}
	\norm{y_{s+1}-y^*}\leq \norm{y_s-y^*}.
\end{equation}
Hence, $\left\{y_s\right\}_{s=0}^\infty$ is a bounded sequence, therefore there is a converging subsequence\\ $\left\{y_{s_k}\right\}_{k=0}^\infty\;:\; \lim\limits_{k\rightarrow\infty}y_{s_k}=\bar{y}$. Then, from (\ref{eq:4.23} (a)) we have $\lim\limits_{k\rightarrow\infty}\hat{y}_{s_k}=\bar{y}$ and from (\ref{eq:4.23}(b)) we obtain $\lim\limits_{k\rightarrow\infty}y_{s_k+1}=\bar{y}$.\\
\par Keeping in mind continuity $g$ we obtain 
$$\bar{y}=\lim\limits_{k\rightarrow\infty}y_{s_k+1}=\lim\limits_{k\rightarrow\infty}\left[y_{s_k}+tg\left(\hat{y}_{s_k}\right)\right]_+=\left[\bar{y}+tg\left(\bar{y}\right)\right]_+.$$
From (\ref{eq:3.2}) follows $\bar{y}=y^*$, which together with (\ref{eq:4.24}) leads to $\lim\limits_{s\rightarrow\infty}y_s=y^*$.
\end{proof}
\section{Convergence rate of the EPG method}
At this point we replace the assumptions (\ref{eq:2.1})-(\ref{eq:2.3}) by less restrictive assumptions (\ref{eq:2.21}) - (\ref{eq:2.23}).
From (\ref{eq:2.21})-(\ref{eq:2.23}) and Lemma \ref{lm:2} follows 
\begin{equation}\label{eq:5.1}
	\left<g(y)-g(y^*), y-y^*\right>\leq -\delta \norm{y-y^*}^2
\end{equation}
with $\delta=\min\{\alpha,\beta,\gamma\}>0$.\\
Therefore, 
$$\left<g(y), y-y^*\right>-\left<g(y^*), y-y^*\right>\leq -\delta\norm{y-y^*}^2,\;\;\;\forall y\in\Omega.$$
Then, $\left(g(y^*), y-y^*)\right)\leq 0,\;\;\;\forall y\in\Omega$, therefore 
\begin{equation}\label{eq:5.2}
	\left<g(y), y-y^*\right>\leq -\delta\norm{y-y^*}^2,\;\;\;\forall y\in\Omega.
\end{equation}
\begin{theorem}\label{th:5}
	If assumptions (\ref{eq:2.21}) - (\ref{eq:2.23}) are satisfied and Lipschitz condition (\ref{eq:3.16}) holds, then for $\nu(t)=1+2\delta t-2(tL)^2$ and ratio $q(t)=1-2\delta t +4(\delta t)^2(\nu(t))^{-1}$ the following statements are taking place:
	\begin{enumerate}
		\item $\norm{y_{s+1}-y^*}\leq\sqrt{q(t)}\norm{y_s-y^*},\;\;0<q(t)<1,\;\forall t\in\left(0,\left(\sqrt{2}L\right)^{-1}\right);$
		\item for $t=(2L)^{-1}$ we have 
		$$q\left(\left(2L\right)^{-1}\right)=(1+\varkappa)(1+2\varkappa)^{-1};$$
		\item for any $\varkappa\in[0,0.5]$ we obtain 
		\begin{equation}\label{eq:5.3}
			\norm{y_{s+1}-y^*}\leq \sqrt{1-0.5\varkappa}\norm{y_s-y^*};
		\end{equation}
	\item 
	\begin{equation}\label{eq:5.4}
		\text{comp} (EPG)\leq \mathcal{O}\left(n^2\varkappa^{-1}\ln\epsilon^{-1}\right),
	\end{equation}
where $\epsilon>0$ is the required accuracy.
	\end{enumerate}

\end{theorem}
\begin{proof}
	\begin{enumerate}
		\item 

	From (\ref{eq:4.1})-(\ref{eq:4.2}), non-expansiveness of operator $P_\Omega$ and Lipschitz condition (\ref{eq:3.16}) follows 
	$$\norm{\hat{y}_s-y_s}=\norm{P_\Omega(y_s+tg(y_s))-P_\Omega(y_s+tg(\hat{y}_s))}
	\leq t\norm{g(y_s)-g(\hat{y}_s)}\leq tL\norm{y_s-\hat{y}_s}.$$
	Using argument from the proof of Theorem \ref{th:4}, we obtain
	\begin{equation}\label{eq:5.5}
		\left<y_s-y_{s+1}, y^*-y_{s+1}\right>+\left<y_s-\hat{y}_s, y_{s+1}-\hat{y}_s\right>+t\left(g(\hat{y}_s), y^*-\hat{y}_s\right)-(Lt)^2\norm{\hat{y}_s-y_s}^2\leq 0.
	\end{equation}
From (\ref{eq:5.2}) with $y=\hat{y}_s$ we obtain
$$\left<g(\hat{y}_s), y^*-\hat{y}_s\right>\geq \delta \norm{\hat{y}_s-y^*}.$$
Therefore, from (\ref{eq:5.5}) follows 
\begin{equation}\label{eq:5.6}
	2\left<y_s-y_{s+1}, y^*-y_{s+1}\right>+2\left<y_s-\hat{y}_s, y_{s+1}-\hat{y}_s\right>+2\delta t\norm{\hat{y}_s-y^*}^2-2(Lt)^2\norm{\hat{y}_s-y_s}\leq 0.
\end{equation}
By applying identity (\ref{eq:4.17}) to the scalar products in (\ref{eq:5.6}) first with $u=y_s, v=y_{s+1}, w=y^*$ and then with $u=y_s, v=\hat{y}_s$, $w=y_{s+1}$, we obtain
\begin{multline*}
\norm{y_s-y_{s+1}}^2+\norm{y_{s+1}-y^*}^2-\norm{y_s-y^*}^2+\norm{y_s-\hat{y}_s}^2\\+
\norm{\hat{y}_s-y_{s+1}}^2-\norm{y_s-y_{s+1}}^2+2\delta t \norm{\hat{y}_s-y^*}^2
-2(Lt)^2\norm{y_s-\hat{y}_s}^2\leq 0
\end{multline*}
or 
\begin{equation}\label{eq:5.7}
	\norm{y_{s+1}-y^*}^2+\norm{\hat{y}_s-y_{s+1}}^2+\left(1-2(Lt)^2\right)\norm{y_s-\hat{y}_s}^2+2\delta t\norm{\hat{y}_s-y^*}^2\leq \norm{y_s-y^*}^2.
\end{equation}
From 
$$\norm{\hat{y}_s-y^*}^2=\left<\hat{y}_s-y_s+y_s-y^*, \hat{y}_s-y_s+y_s-y^*\right>=
\norm{\hat{y}_s-y_s}^2+2\left<\hat{y}_s-y_s, y_s-y^*\right>+\norm{y_s-y^*}^2$$
and (\ref{eq:5.7}) follows 
\begin{multline}\label{eq:5.8}
	\norm{y_{s+1}-y^*}^2+\norm{\hat{y}_s-y_{s+1}}^2+\left(1+2\delta t-2(Lt)^2\right)\norm{\hat{y}_s-y_s}^2-4\delta t \left<\hat{y}_s-y_s, y_s-y^*\right>\\\leq (1-2\delta t)\norm{y_s-y^*}^2.
\end{multline}
By introducing $\nu(t)=1+2\delta t - 2(Lt)^2$ we can rewrite the third and fourth terms of left hand side of (\ref{eq:5.8}) as follows: 
$$\norm{\sqrt{\nu(t)}(\hat{y}_s-y_s)+2(y_s-y^*)\frac{\delta t}{\sqrt{\nu(t)}}}^2-\frac{4(\delta t)^2 \norm{y_s-y^*}^2}{\nu (t)}.$$
Then, from (\ref{eq:5.8}) follows
\begin{multline}\label{eq:5.9}
 \norm{y_{s+1}-y^*}^2+\norm{\hat{y}_s-y_{s+1}}^2+\norm{\sqrt{\nu(t)}(\hat{y}_s-y_s)+2(y_s-y^*)\frac{\delta t}{\sqrt{\nu(t)}}}^2\\\leq \left(1-2\delta t+\frac{4(\delta t)^2}{\nu(t)}\right)\norm{y_s-y^*}^2.
\end{multline} 

Hence, for $q(t)=1-2\delta t+4(\delta t)^2(\nu(t))^{-1}$ from (\ref{eq:5.9}) follows 
\begin{equation}\label{eq:5.10}
	\norm{y_{s+1}-y^*}^2\leq q(t)\norm{y_s-y^*}^2.
\end{equation} 
and $0<q(t)<1,\;\forall t\in \left(0,\left(\sqrt{2}L\right)^{-1}\right)$.
\item 
For $t=(2L)^{-1}$ and $\varkappa=\delta L^{-1}$ we obtain 
\begin{equation}\label{eq:5.11}
	q\left((2L)^{-1}\right)=1-\varkappa+\varkappa^2(0.5+\varkappa)^{-1}=(1+\varkappa)(1+2\varkappa)^{-1}.
\end{equation}
For any $\varkappa\in [0,0.5]$ we have $(1+\varkappa)(1+2\varkappa)\leq 1-0.5\varkappa$, therefore from (\ref{eq:5.10}) and (\ref{eq:5.11}) follows (\ref{eq:5.3}).
\item 
From (\ref{eq:5.3}), $q=\sqrt{1-0.5\varkappa}, \varkappa\in[0,0.5]$ and given accuracy $0<\epsilon\ll 1$ follows that EPG method requires 
$$s=\mathcal{O}\left(\ln \epsilon^{-1}\left(\ln q^{-1}\right)^{-1}\right)$$
steps to find an approximation with accuracy $\epsilon$. Then,
\begin{equation}\label{eq:5.12}
	\left(\ln q^{-1}\right)^{-1}=\left(-\frac{1}{2}\ln (1-0.5\varkappa)\right)^{-1}.
\end{equation}
From $\ln(1+x)\leq x\;\;\;\forall x>-1$ follows $\ln(1-0.5\varkappa)\leq -0.5\varkappa$, therefore from (\ref{eq:5.12}) we obtain $\left(\ln q^{-1}\right)^{-1}\leq 4\varkappa ^{-1}$. Keeping in mind that each step requires $\mathcal{O}(n^2)$ operations we obtain complexity bound (\ref{eq:5.4}).
	\end{enumerate}
\end{proof}
\par \textbf{Remark 2}. For small $\varkappa>0$ complexity  bound (\ref{eq:5.4}) is much better then bound (\ref{eq:3.12}). Therefore, although EPG requires two projection on $\Omega$ instead of one for PGP, the EPG is still more efficient. For very large NPCE problems the PGP can compete with the EPG.
\section{Concluding remarks}
\par The NPCE is a fundamental departure from both LP and IO models, when it comes to optimal distribution limited resources. It became possible by replacing the fixed production cost, fixed consumption and fixed factors availability with production, consumption and factor operators.
\par It leads to data "symmetrization" and allows to replace finding NPCE by solving a VI with a very simple feasible set $\Omega$, projection on which is a low cost operation. Therefore, for solving VI we use two projection methods PGP and EPG, which do not require solving a QP or a system of linear equations at each step. Instead, they require at each step few matrix by vector multiplication.
\par Moreover, the primal and dual variables one computes independently, in parallel, that is both methods decompose the problem, which allows solving large scale NPCE.
\par On the other hand, both method are nothing, but pricing mechanisms for finding NPCE. 
\par From (\ref{eq:1.3}) follows that at NPCE the total consumption is equal to the sum of production and factors costs.
\par From (\ref{eq:1.4})-(\ref{eq:1.6}) follows that at the NPCE the market clearing takes place.
\par From (\ref{eq:3.12}) and (\ref{eq:5.4}) follows that only three parameters are responsible for complexity of any given NPCE problem: the condition number of the VI operator, the size of the problem and the given accuracy. 
It follows from the complexity bounds that in some instances finding NPCE can be easier than solving LP of the same size. 
\par The numerical results obtained so far support the theoretical findings. The practical use of NPCE, however, will require much more work for developing the input data, numerical realization of the PGP and EPG methods, conducting numerical experiments as well as careful analyzing the results obtained.

	\bibliographystyle{plain} % We choose the "plain" reference style
\bibliography{refs} % Entries are in the refs.bib file
\end{document}